\documentclass{amsart}

\usepackage{amsfonts,amssymb,amsopn,amsmath}
\usepackage{mleftright}
\usepackage{arydshln}
\usepackage{graphicx}
\usepackage{url}
\usepackage{epstopdf}
\usepackage{extarrows}
\usepackage{braket}
\usepackage{color}
\usepackage[nocompress]{cite}
\usepackage{subfigure}
\usepackage{enumerate}
\usepackage{cleveref}

\DeclareMathOperator{\rank}{rank}
\DeclareMathOperator{\corank}{corank}
\DeclareMathOperator{\codim}{codim}
\DeclareMathOperator{\coker}{coker}
\DeclareMathOperator{\img}{Im}

\newcommand{\bigzerol}{\smash{\lower1.0ex\hbox{\bg 0}}}

\newcommand{\Crit}{\mathrm{Crit}}

\newcommand{\N}{\mathbb{N}}
\newcommand{\R}{\mathbb{R}}
\newcommand{\gen}[1]{\mleft\langle#1\mright\rangle}
\newcommand{\prn}[1]{\mleft(#1\mright)}

\newcommand{\norm}[1]{\mleft\|#1\mright\|}

\newfont{\bg}{cmr9 scaled\magstep4}

\theoremstyle{plain}
\newtheorem{theorem}{Theorem}[section]

\newtheorem{lemma}[theorem]{Lemma}
\newtheorem{proposition}[theorem]{Proposition}

\theoremstyle{definition}
\newtheorem{definition}[theorem]{Definition}
\newtheorem{example}[theorem]{Example}

\newtheorem{remark}[theorem]{Remark}

\newcommand{\TheTitle}{Topology of Pareto sets of strongly convex problems}

\title{{\TheTitle}}

\author[N.~Hamada]{Naoki Hamada}
\address{Artificial Intelligence Laboratory, Fujitsu Laboratories Ltd., Kawasaki, 211-8588, Japan;
  RIKEN AIP Fujitsu Collaboration Center, RIKEN Center for Advanced Intelligence Project, Tokyo, 103-0027, Japan.} \email{hamada-naoki@fujitsu.com}
\author[K.~Hayano]{Kenta Hayano}
\address{Department of Mathematics Faculty of Science and Technology, Keio University, Yokohama, 223-8522, Japan; 
Mathematical Science Team, RIKEN Center for Advanced Intelligence Project, Tokyo, 103-0027, Japan.}\email{k-hayano@math.keio.ac.jp}
\author[S.~Ichiki]{Shunsuke Ichiki}
\address{Institute of Mathematics for Industry, Kyushu University, Fukuoka, 819-0395, Japan.} \email{s-ichiki@imi.kyushu-u.ac.jp}

  \author[Y.~Kabata]{Yutaro Kabata}
\address{Institute of Mathematics for Industry, Kyushu University, Fukuoka, 819-0395, Japan.}\email{kabata@imi.kyushu-u.ac.jp}
 \author[H.~Teramoto]{Hiroshi Teramoto}
\address{Molecule \& Life Nonlinear Sciences Laboratory, Research Institute for Electronic Science, Hokkaido University, Sapporo 001-0020, Japan; JST, PRESTO, Department of Research Promotion, Tokyo, 102-0076, Japan.}\email{teramoto@es.hokudai.ac.jp}

\begin{document}

\maketitle

\begin{center}
\small\emph{Dedicated to Professor Takashi Nishimura on the occasion of his 60th birthday}    
\end{center}

\begin{abstract}
A multiobjective optimization problem is simplicial if the Pareto set and front are homeomorphic to a simplex and, under the homeomorphisms, each face of the simplex corresponds to the Pareto set and front of a subproblem.
In this paper, we show that strongly convex problems are simplicial under a mild assumption on the ranks of the differentials of the objective mappings.
We further prove that one can make any strongly convex problem satisfy the assumption by a generic linear perturbation, provided that the dimension of the source is sufficiently larger than that of the target.
We demonstrate that the location problems, a biological modeling, and the ridge regression can be reduced to multiobjective strongly convex problems via appropriate transformations preserving the Pareto ordering and the topology.
\end{abstract}

%\begin{keywords}
%multiobjective optimization, strongly convex mapping, simplicial problem, topology of differentiable mapping
%\end{keywords}

%\begin{AMS}
%  90C25, 57R35, 57R45
%\end{AMS}

%%%%%%%%%%%%%%%%%%%%%%%%%%%%%%%%%%%%%%%%%%%%%%%%%%%%%%%%%%%%%%%%%%%%%%%%%%%%%%%%
\section{Introduction}\label{sec:introduction}
Multiobjective optimization arises in various fields of science and engineering (e.g., data mining \cite{Mukhopadhyay2014a,Mukhopadhyay2014b}, finance \cite{Ponsich2013}, car design \cite{Tayarani2015}, and more \cite{Zhou2011}).
A common scenario in the classical decision making is that users first specify their preference, then scalarize objective functions according to the preference, and finally solve a scalarized problem to find the preferred Pareto solution \cite{Miettinen1999,Ehrgott2005}.
In contrast to this a priori approach, the recent computational power enables us to take an a posteriori approach: users first solve a multiobjective problem to obtain the whole Pareto set (or an approximation of it), then consider the trade-off between objective functions, and finally make a choice of the preferred solution.
In the a posteriori approach, some multiobjective solvers utilize scalarization for guiding search directions in order to obtain the entire Pareto set (see for example \cite{Hillermeier2001,Zhang2007,Eichfelder2008,Deb2014}).

In general it becomes easier to obtain the whole Pareto set if it has a simple topological structure.
For example, applying a generalized homotopy method \cite{Hillermeier2001} one can find the entire \emph{connected} Pareto set from a single Pareto solution.
Furthermore, we can efficiently compute a parametric-surface approximation of the entire Pareto set, provided that the problem is \emph{simplicial}~\cite{Kobayashi2019}.
\Cref{fig:simplicial-problem} depicts an example of a simplicial problem with the three objective functions $f_1, f_2, f_3$.
As this figure indicates, there exists a homeomorphism $\Phi$ from a simplex $\Delta^2$ to the Pareto set, sending each face of $\Delta^2$ to the Pareto set of a subproblem (for the precise definition of simplicial problems, see \cref{sec:simplicial problem}).
The property that $\Phi$ sends the faces to the Pareto sets has a great utility: it is known that using this correspondence, we can reduce the number of solutions to obtain an approximation of the mapping $\Phi$, i.e., a parametric-surface approximation of the Pareto set (cf.~\cite{Kobayashi2019,Tanaka2019}).

In the literature, some researchers also studied hierarchical structures of solution sets of various types.
Shoval et al.\ pointed out that when minimizing distances from given points, one can decompose the boundary of the Pareto set into the Pareto sets of subproblems \cite{Shoval2012}.
Mainly based on Wan's results \cite{Wan1975,Wan1977}, Lovison and Pacci showed a generic property of mappings that the local Pareto set admits a Whitney stratification \cite{Lovison2014}.
Gebken et al.\ investigated a situation that the Pareto critical set can be determined by a reduced number of objectives when the number of objectives are greater than that of variables~\cite{Gebken2019}.
Compared to the above studies, the theory of simplicial problems uniquely provides efficient computation of the whole Pareto set of some class of mappings from high-dimensional decision spaces to low-dimensional objective spaces, which are frequently seen in practical problems.

It is therefore important to investigate what kinds of known problem classes, such as linear problems, convex problems\footnote{While it has been stated that \emph{any} convex problem is simplicial in the literature, it is too optimistic to expect so, as one can easily find a counterexample.
See e.g.~\cref{S:example1} and \cref{R:assumption dimensions}.}, etc., are simplicial.
In this paper, we will show that problems minimizing strongly convex mappings are simplicial under the assumption on the differentials of the mappings.
\begin{figure}[tbp]
    \centering
    \includegraphics[width=\textwidth]{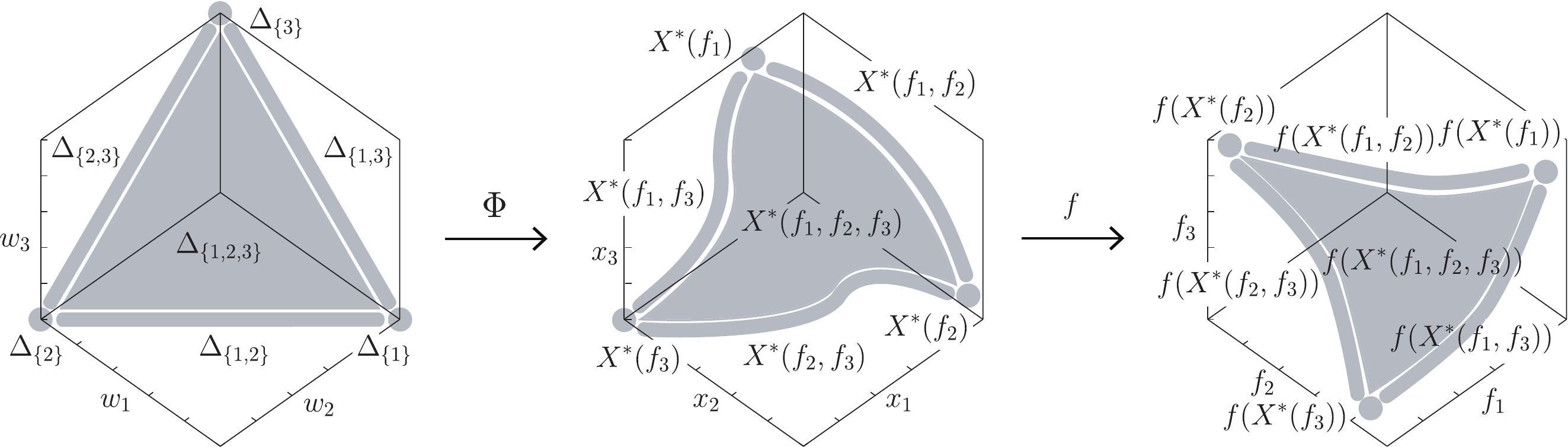}
    \caption{The Pareto set $X^*(f)$ (middle) and front $f(X^*(f))$ (right) of a simplicial problem $f = (f_1, f_2, f_3)$.
    There exists a homeomorphism $\Phi$ from a simplex $\Delta^2$ to $X^*(f)$, sending each face to the Pareto set of a subproblem. The composite $f \circ \Phi: \Delta^2 \to f(X^*(f))$ is also a homeomorphism.}
    \label{fig:simplicial-problem}
\end{figure}

\begin{theorem}\label{T:main theorem}
Let $f: \R^n \to \R^m$ be a strongly convex $C^r$--mapping $(2 \le r \le \infty)$.
The multiobjective optimization problem minimizing $f$ is $C^{r - 1}$--weakly simplicial.
Furthermore, this problem is $C^{r - 1}$--simplicial if the corank of the differential $d f_x$ is equal to $1$ for any $x \in X^*(f)$.
\end{theorem}
For the definitions of strong convexity and weak simpliciality, see \cref{sec:preliminaries}.
Under the assumptions of \cref{T:main theorem}, the weighted-sum scalarization gives an instance of the aforementioned homeomorphism $\Phi$, that is, the mapping $x^*: \Delta^{m - 1} \to X^*(f)$ defined by
\[
    x^*(w) = \arg\min_{x \in \R^n} \prn{\sum_{i=1}^m w_i f_i(x)}
\]
is a homeomorphism.
Since the scalarized function $\sum_{i=1}^m w_i f_i: \R^n \to \R$ is strongly convex (see \cref{S:chara-pareto}), the mapping $x^*$ is well-defined and can be efficiently computed by, e.g., gradient methods.
By this property, scalarization-based subdivision methods \cite{Eichfelder2008,Mueller-Gritschneder2009,Hamada2010,Singh2011} are able to provide solutions for parametric-surface approximation methods in \cite{Kobayashi2019,Tanaka2019} to build an approximation of $\Phi$.

Note that (strict) convexity of a mapping does not necessarily imply that the corresponding problem is simplicial.
For example, the single objective problem minimizing the function $\exp(x)$ (defined on $\R$) does not have a Pareto solution (i.e.~a minimizer), in particular it is not simplicial, although $\exp(x)$ is strictly convex.

As the example given in \cref{S:example2} indicates, we cannot drop the assumption on the corank in \cref{T:main theorem}.
Nevertheless, one can make any strongly convex mapping satisfy the corank assumption by a generic linear perturbation, provided that the dimension of the source is sufficiently larger than that of the target.
See \cref{T:Pareto_convex generic} for details.
We will further observe that \cref{T:Pareto_convex generic} would not hold without the assumption on the dimensions (cf.~\cref{R:assumption dimensions}).
Note that a small linear perturbation of a strongly convex problem does not cause substantial changes of the Pareto set (see \cref{rem:Gamma} for details).

Applying \cref{T:main theorem}, we can show that several problems appearing in practical situations are simplicial.
In 1967, Kuhn~\cite{Kuhn1967} showed that the Pareto set of a location problem under the Euclidean norm is the convex hull of demand points, which becomes a simplex when the demand points are in general position.
We will observe in \cref{sec:applications} that this problem can be made strongly convex by suitable transformations on the target space preserving the Pareto ordering. 
In particular, our theorem gives an alternative proof of Kuhn's result. 
Shoval~et~al.~\cite{Shoval2012} proposed a multiobjective model for phenotypic divergence of species in evolutionary biology and indicated that its Pareto set is a curved simplex. 
Again, this problem becomes strongly convex after suitable transformations (see \cref{sec:applications}). 
We can thus give a rigorous proof for the observation in~\cite{Shoval2012}. 

The paper is organized as follows:
After preliminaries in \cref{sec:preliminaries}, the proof of the main theorem will be given in \cref{S:Topology Paretoset}.
In \cref{sec:generic}, we will show that, under some assumption on the dimensions of the source and the target, any strongly convex problem becomes simplicial after a generic linear perturbation.
\Cref{sec:applications} will be devoted to discussing practical problems.

%%%%%%%%%%%%%%%%%%%%%%%%%%%%%%%%%%%%%%%%%%%%%%%%%%%%%%%%%%%%%%%%%%%%%%%%%%%%%%%%
\section{Preliminaries}\label{sec:preliminaries}
%%%%%%%%%%%%%%%%%%%%%%%%%%%%%%%%%%%%%%%%%%%%%%%%%%%%%%%%%%%%%%%%%%%%%%%%%%%%%%%%
We introduce the definition of strongly convex problems and their properties and define $C^r$--(weakly) simplicial problems.
Throughout the paper, we denote the index set $\Set{1, \dots, m}$ by $M$.

%%%%%%%%%%%%%%%%%%%%%%%%%%%%%%%%%%%%%%%%%%%%%%%%%%%%%%%%%%%%%%%%%%%%%%%%%%%%%%%%
\subsection{Multiobjective optimization}
A multiobjective optimization problem is a problem minimizing objective functions $f_1, \dots, f_m:X \to \R$ over a subset $X \subseteq \R^n$:
\begin{align*}
    \text{minimize } & f(x) = (f_1(x), \dots, f_m(x))\\
    \text{subject to } & x \in X (\subseteq \R^n).
\end{align*}
According to the \emph{Pareto ordering}, i.e.,
\begin{align*}
    f(x) \prec f(y) \xLeftrightarrow{\text{def}}
    f_i(x) \le f_i(y) \text{ for all } i \in M \text{ and }
    f_j(x) < f_j(y) \text{ for some } j \in M,
\end{align*}
we basically would like to obtain the \emph{Pareto set}
\[
X^*(f) = \Set{x \in X | \forall y \in X,\ f(y) \not \prec f(x)}
\]
and the \emph{Pareto front}
\[
f(X^*(f)) = \Set{f(x) \in \R^m | x \in X^*(f)}.
\]

%%%%%%%%%%%%%%%%%%%%%%%%%%%%%%%%%%%%%%%%%%%%%%%%%%%%%%%%%%%%%%%%%%%%%%%%%%%%%%%%
\subsection{Simplicial problems}\label{sec:simplicial problem}
Here, we explain the definition of $C^r$--(weakly) simplicial problems for $0 \le r \le \infty$.
For $\varepsilon \ge 0$, we define the subset $\Delta_\varepsilon^{m - 1} \varsubsetneq \R^m$ as follows:
\[
\Delta_\varepsilon^{m - 1} = \Set{(w_1, \dots, w_m) \in \R^m | \sum_{i = 1}^m w_i = 1,\ w_i > - \varepsilon} \varsubsetneq \R^m.
\]
Note that the closure $\overline{\Delta_0^{m - 1}}$ is the standard simplex, which we will denote by
\[
\Delta^{m - 1} = \Set{(w_1, \dots, w_m) \in \R^m | \sum_{i = 1}^m w_i = 1,\ w_i \ge 0} \varsubsetneq \R^m.
\]
We also denote a face of $\Delta^{m - 1}$ for $I \subseteq M$ by
\[
\Delta_I = \Set{(w_1, \dots, w_m) \in \Delta^{m - 1} | w_i = 0\ (i \not \in I)} \varsubsetneq \R^m.
\]
For a subset $U \subseteq \R^m$, a continuous mapping $f: \Delta^{m - 1} \to U$ is a \emph{$C^r$--mapping} if there exist $\varepsilon>0$ and a $C^r$--mapping $\tilde{f}: \Delta^{m - 1}_\varepsilon \to \R^m$ satisfying $\tilde{f}|_{\Delta^{m - 1}} = f.$\footnote{The usual definition of a $C^r$--mapping on a manifold with corners is slightly different from that given here: the latter is stronger (as a condition) than the former.}
A subspace $X \subseteq \R^n$ is \emph{$C^r$--diffeomorphic} to the simplex $\Delta^{m - 1}$ if there exist $\varepsilon >0$ and a $C^r$--immersion $\phi: \Delta_\varepsilon^{m - 1} \to \R^n$ such that $\phi|_{\Delta^{m - 1}}: \Delta^{m - 1} \to X$ is a homeomorphism.
The reader can refer to \cite[\S.2]{Michor1980} for more general definition of diffeomorphisms between manifolds with corners.

\begin{definition}
Let $X$ be a subset of $\R^n$ and $f = (f_1, \dots, f_m)$ be a mapping from $X$ to $\R^m$.
For $I = \Set{i_1, \dots, i_k} \subseteq M$ such that $i_1 < \dots < i_k$, we put $f_I = (f_{i_1}, \dots,  f_{i_k})$.
The problem minimizing $f$ is \emph{$C^r$--simplicial} if there exists a $C^r$--mapping $\Phi: \Delta^{m - 1} \to X^*(f)$ such that both of the restrictions $\Phi|_{\Delta_I}: \Delta_I \to X^*(f_I)$ and $f|_{X^*(f_I)}$ are $C^r$--diffeomorphisms for any $I \subseteq M$.
The problem minimizing $f$ is \emph{$C^r$--weakly simplicial} if there exists a $C^r$--mapping $\phi: \Delta^{m - 1} \to f(X^*(f))$ satisfying $\phi(\Delta_I) = f(X^*(f_I))$ for any $I \subseteq M$.
\end{definition}

\subsection{Pareto solutions of strongly convex mappings}\label{S:chara-pareto}
In this subsection, a characterization of Pareto solutions of strongly convex $C^1$--mappings is given (see \cref{T:chara-pareto}).
We begin this subsection with quickly reviewing the definition of (strong) convexity.
A subset $X$ of $\R^n$ is \emph{convex} if $tx + (1 - t)y \in X$ for all $x, y \in X$ and all $t \in [0, 1]$.
Let $X$ be a convex set in $\R^n$.
A function $f: X \to \R$ is \emph{convex} if 
\[
f(tx + (1 - t)y) \le tf(x) + (1 - t)f(y)
\] 
for all $x, y \in X$ and all $t \in [0, 1]$.
A function $f: X \to \R$ is \emph{strongly convex} if there exists $\alpha > 0$ such that
\[
f(tx + (1 - t)y) \le tf(x) + (1 - t)f(y) - \frac{1}{2} \alpha t(1 - t) \norm{x - y}^2
\] 
for all $x, y \in X$ and all $t \in [0, 1]$, where $\norm{x - y}$ denotes the Euclidean norm of $x - y$.
A mapping $f = (f_1, \dots, f_m): X \to \R^m$ is \emph{(strongly) convex} if every $f_i$ is (strongly) convex.
A problem minimizing a strongly convex mapping is called a \emph{strongly convex problem}.
The followings are basic properties of (strongly) convex mappings which will be needed later on:

\begin{lemma}[{\cite[Theorem~2.1.2~(p.~54)]{Nesterov2004}}]\label{T:convexchara}
Let $X \subseteq \R^n$ be a convex open subset.
A $C^1$--function $f: X \to \R$ is convex if and only if $f(x) + d f_x \cdot (y - x) \le f(y)$ for any $x , y \in X$.
\end{lemma}

\begin{lemma}[{\cite[Theorem~2.1.11~(p.~65)]{Nesterov2004}}]
\label{T:criterion storngly convex Hessian}
Let $X \subseteq \R^n$ be a convex open subset.
A $C^2$--function $f: X \to \R$ is strongly convex if and only if there exists $\beta > 0$ such that $m(f)_x \ge \beta$ for any $x \in X$, where $m(f)_x$ is the minimal eigenvalue of the Hessian matrix of $f$ at $x$.
\end{lemma}

\begin{lemma}[{\cite[Theorem~2.2.6~(p.~85)]{Nesterov2004}}]\label{T:minimum}
Let $f: \R^n \to \R$ be a strongly convex $C^1$--function.
Then, there exists a unique point such that the function $f$ is minimized.
\end{lemma}

In the rest of this subsection we will prove the following proposition:
\begin{proposition}\label{T:chara-pareto}
Let $f = (f_1, \dots, f_m): \R^n \to \R^m$ be a strongly convex $C^1$--mapping.
Then, $x \in X^*(f)$ if and only if there exists an element $(w_1, \dots, w_m) \in \Delta^{m - 1}$ such that $f$ satisfies one (and hence both) of the following equivalent conditions:
  \begin{enumerate}
  \item $\sum_{i = 1}^m w_i (d f_i)_x = 0$.
  \item The point $x \in \R^n$ is a unique element such that the function $\sum_{i = 1}^m w_i f_i$ is minimized.
  \end{enumerate}
\end{proposition}
For the proof, we prepare some lemmas.
\begin{lemma}[{\cite[Theorem~3.1.3~(p.~79)]{Miettinen1999}}]\label{T:sufficient}
Let $f = (f_1, \dots, f_m): \R^n \to \R^m$ be a mapping and let $(w_1, \dots, w_m) \in \Delta^{m - 1}$ be an element.
If $x \in \R^n$ is a unique element such that the function $\sum_{i = 1}^m w_i f_i$ is minimized, then $x \in X^*(f)$.
\end{lemma}
The following is a special case of the Karush-Kuhn-Tucker necessary condition for Pareto optimality.
\begin{lemma}[{\cite[Theorem~3.1.5~(p.~39)]{Miettinen1999}}]\label{T:necessary}
Let $f = (f_1, \dots, f_m): \R^n \to \R^m$ be a $C^1$--mapping.
If $x \in X^*(f)$, then there exists an element $(w_1, \dots, w_m) \in \Delta^{m - 1}$ satisfying $\sum_{i = 1}^m w_i (d f_i)_x = 0$.
\end{lemma}
\begin{lemma}\label{T:convex_property}
Let $f = (f_1, \dots, f_m): \R^n \to \R^m$ be a convex $C^1$--mapping.
Let $(w_1, \dots, w_m) \in \Delta^{m - 1}$ be an element.
Then, the following conditions for $x \in \R^n$ are equivalent.
  \begin{enumerate}
  \item $\sum_{i = 1}^m w_i (d f_i)_x = 0$.
  \item The function $\sum_{i = 1}^m w_i f_i: \R^n \to \R$ attains its minimum at $x$.
  \end{enumerate}
\end{lemma}
\begin{proof}[Proof of \cref{T:convex_property}]
Set $g = \sum_{i = 1}^m w_i f_i$.
Then, for any $x \in \R^n$, we have
  \begin{align}
  \sum_{i = 1}^m w_i (d f_i)_x = dg_x.
  \label{Eq:convex_gradient_1}
  \end{align}
Since $g$ is convex, we can deduce from \cref{T:convexchara} that the following inequality holds for any $y \in \R^n$:
  \begin{align}
  g(x) + dg_x \cdot (y-x) \le g(y).
  \label{Eq:convex_gradient_2}
  \end{align}
Suppose that $\sum_{i = 1}^m w_i (d f_i)_x = 0$.
We can easily deduce the assertion~2 from \cref{Eq:convex_gradient_1} and \cref{Eq:convex_gradient_2}.
Suppose that $\sum_{i = 1}^m w_i f_i: \R^n \to \R$ attains its minimum at $x$.
Since $dg_x$ is equal to $0$ and the equality \cref{Eq:convex_gradient_1} holds, we have the assertion~1.
\end{proof}

\begin{proof}[Proof of \cref{T:chara-pareto}]
Suppose that $x \in X^*(f)$.
Using \cref{T:necessary}, we can verify that there exists an element $(w_1, \dots, w_m) \in \Delta^{m - 1}$ satisfying $\sum_{i = 1}^m w_i (d f_i)_x = 0$.
From \cref{T:convex_property}, the point $x \in \R^n$ is an element such that $\sum_{i = 1}^m w_i f_i$ is minimized.
Since $\sum_{i = 1}^m w_i f_i$ is strongly convex $C^1$--function, by \cref{T:minimum}, we have the assertion~2.
Finally, suppose the assertion~2.
Then, from \cref{T:sufficient}, we get $x \in X^*(f)$.
\end{proof}

%%%%%%%%%%%%%%%%%%%%%%%%%%%%%%%%%%%%%%%%%%%%%%%%%%%%%%%%%%%%%%%%%%%%%%%%%%%%%%%%
\subsection{Fold singularities}\label{S:fold}
In this subsection we will briefly review the definition and basic properties of fold singularities (for details, see \cite{Golubitsky1974}).
For $0 \le k \le \min \Set{n, m}$, we define a subset $S_k \varsubsetneq J^1(\R^n, \R^m)$ as follows:
\[
S_k = \Set{j^1 g(x) \in J^1(\R^n, \R^m) | \begin{minipage}[c]{56mm}
$x \in \R^n$, $g: \R^n \to \R^m$~:~$C^2$--mapping,
$\min \Set{n, m} - \rank(dg_x) = k$
\end{minipage}},
\]
where $j^1 g: \R^n \to J^1(\R^n, \R^m)$ is the $1$--jet extension of $g$.
Let $g: \R^n \to \R^m$ be a $C^2$--mapping, $S \subseteq J^1(\R^n, \R^m)$ be a submanifold, and $x \in \R^n$.
The mapping $j^1 g$ is \emph{transverse} to $S$ at $x$ if either of the following conditions holds:
\begin{itemize}
  \item $j^1 g(x)$ is not contained in $S$, 
  \item $j^1 g(x) \in S$ and $ d(j^1 g)_x (T_x \R^n) + T_{j^1 g(x)} S = T_{j^1 g(x)} J^1(\R^n, \R^m)$.
\end{itemize}
The mapping $j^1 g$ is \emph{transverse} to $S$ if it is transverse to $S$ at any point in $\R^n$.

Suppose that $n$ is greater than or equal to $m$.
For a $C^2$--mapping $f: \R^n \to \R^m$, we denote the critical point set of $f$ by $\Crit(f) \subseteq \R^n$.
A point $x \in \Crit(f)$ is called a \emph{fold} if the following conditions hold:
\begin{enumerate}
  \item 
  $j^1 f$ is transverse to $S_1$ at $x_0$.
  \item 
  $T_{x_0} S_1(f) \oplus \ker d f_{x_0} = T_{x_0} \R^n$, where $S_1(f) = (j^1 f)^{-1}(S_1)$.
\end{enumerate}
Note that we can easily deduce from the condition 2 that the restriction $f|_{\Crit(f)}$ is an immersion around a fold.

\begin{remark}
Let $f: \R^n \to \R^m$ be a $C^\infty$--mapping and $x \in \Crit(f)$ be a fold.
One can take coordinate neighborhoods $(U, \varphi)$ and $(V, \psi)$ at $x$ and $f(x)$, respectively, so that they satisfy:
\[
\psi \circ f \circ \varphi^{-1}(x_1, \dots, x_n) = \prn{x_1, \dots, x_{m - 1}, \sum_{k = m}^n \pm x_k^2}.
\]
\end{remark}

In what follows we will give a useful criterion for detecting fold singularities.
Let $f = (f_1, \dots, f_m): \R^n \to \R^m$ be a $C^2$--mapping and $x_0 \in \Crit(f)$.
Suppose that the corank
\footnote{For a linear mapping $\varphi: V \to W$ the non-negative number $\dim W - \rank(\varphi)$ is called the \emph{corank} of $\varphi$.}
of $d f_{x_0}$ is $1$ and the matrix $\prn{\frac{\partial f_i}{\partial x_j}(x_0)}_{1 \le i, j \le m - 1}$ is regular.
We define the function $\lambda_f: \R^n \to \R^{n - m + 1}$ as follows:
\[
\lambda_f(x) = (J_1(x), \dots, J_{n - m + 1}(x)),
\]
where
\[
J_i(x) = \det
\begin{pmatrix}
\dfrac{\partial f_1}{\partial x_1}(x)           & \cdots & \dfrac{\partial f_m}{\partial x_1}(x)\\
\vdots                                          & \ddots & \vdots\\
\dfrac{\partial f_1}{\partial x_{m - 1}}(x)     & \cdots & \dfrac{\partial f_m}{\partial x_{m - 1}}(x)\\
\dfrac{\partial f_1}{\partial x_{m - 1 + i}}(x) & \cdots & \dfrac{\partial f_m}{\partial x_{m - 1 + i}}(x)
\end{pmatrix}.
\]

\begin{lemma}\label{T:criterion fold}
Under the situation above, $x_0$ is a fold if and only if the following conditions hold:
\begin{enumerate}
  \item the differential $(d \lambda_f)_{x_0}$ has rank $n - m + 1$,
  \item $\ker (d \lambda_f)_{x_0} \oplus \ker d f_{x_0} = T_{x_0} \R^n$.
\end{enumerate}
\end{lemma}

\begin{remark}
The two conditions in \cref{T:criterion fold} are equivalent to those in the original definition above.
Indeed, the first condition is equivalent to the condition that $j^1 f$ is transverse to $S_1$ at $x_0$, and $T_{x_0} S_1(f)$ is equal to $\ker (d \lambda_f)_{x_0}$.
\end{remark}

%%%%%%%%%%%%%%%%%%%%%%%%%%%%%%%%%%%%%%%%%%%%%%%%%%%%%%%%%%%%%%%%%%%%%%%%%%%%%%%%
\section{Proof of the main result}\label{S:Topology Paretoset}
%%%%%%%%%%%%%%%%%%%%%%%%%%%%%%%%%%%%%%%%%%%%%%%%%%%%%%%%%%%%%%%%%%%%%%%%%%%%%%%%
In this section we will show that strongly convex problems are simplicial.
Let $f=(f_1, \dots, f_m): \R^n \to \R^m$ be a strongly convex $C^r$--mapping $(2 \le r \le \infty)$.
Since $\sum_{i = 1}^m w_i f_i$ is strongly convex for any $(w_1, \dots, w_m) \in \Delta^{m - 1}$, there exists a unique point $x \in \R^n$ such that $\sum_{i = 1}^m w_i f_i$ is minimized (see \cref{T:minimum}).
We denote this minimizing point by $\arg \min_{x \in \R^n} \prn{\sum_{i = 1}^m w_i f_i(x)} \in \R^n$, which is contained in $X^*(f)$ by \cref{T:sufficient}.
We can thus define a mapping $x^*: \Delta^{m - 1} \to X^*(f)$ as follows:
\[
x^*(w) = \arg \min_{x \in \R^n} \prn{\sum_{i = 1}^m w_i f_i(x)}.
\]

\begin{theorem}\label{T:Pareto_convex}
Let $f = (f_1, \dots, f_m): \R^n \to \R^m$ be a strongly convex $C^r$--mapping $(2 \le r \le \infty)$.
\begin{enumerate}
\item The mapping $x^*: \Delta^{m - 1} \to X^*(f)$ (and thus $f \circ x^*: \Delta^{m - 1} \to f(X^*(f))$) is a surjective mapping of class $C^{r - 1}$.
\item Suppose that the corank of $d f_x$ is equal to $1$ for any $x \in X^*(f)$.
  \begin{enumerate}[A.]
  \item The mapping $x^*: \Delta^{m - 1} \to X^*(f)$ is a $C^{r - 1}$--diffeomorphism.
  \item The restriction $f|_{X^*(f)}: X^*(f) \to \R^m$ is a $C^{r - 1}$--embedding.
  \end{enumerate}
\end{enumerate}
\end{theorem}
Note that this theorem obviously holds for $m = 1$.
For this reason, in the rest of this section we will assume $m \ge 2$.

\cref{T:main theorem} follows from this theorem as follows:
It is easy to see that any subproblem of a strongly convex problem is again strongly convex.
In particular, by applying \cref{T:Pareto_convex} to each subproblem, we can show that the image of the restriction $x^*$ on $\Delta_I$ is equal to $X^*(f_I)$ for any $I \subseteq M$.
Thus a strongly convex problem is weakly simplicial.
We can further deduce from the assertion 2 of \cref{T:Pareto_convex} that a strongly convex problem is simplicial under the assumption on the coranks of differentials.

\begin{remark}\label{R:corank_fold}
The corank assumption in 2 implies that $n$ is greater than or equal to $m - 1$.
As we will show in the proof, under this assumption any point in $X^*(f)$ for a mapping $f$ is a fold\footnote{It is indeed a ``definite'' fold.} if $n \ge m$.
\end{remark}

\begin{remark}
In general, the mapping $x^*$ for a strongly convex problem (without the corank assumption) is not necessarily a diffeomorphism.
We will give an explicit example of such a problem with a non-injective $x^*$ in \cref{S:example1}.
\end{remark}

\begin{proof}[Proof of 1 in \cref{T:Pareto_convex}]
First of all, we can immediately deduce from \cref{T:chara-pareto} that $x^*$ is surjective.
Let $\delta'>0$ be a positive number and $g: \Delta_{\delta'}^{m - 1} \times \R^n \to \R^n$ be a $C^{r - 1}$--mapping defined by $g(w, x) = \sum_{k = 1}^m w_k (d f_k)_x$.
We can easily deduce from \cref{T:convex_property} that $x^*$ is an implicit function of the equation $g(w, x) = 0$ defined on $\Delta^{m - 1}$.
Differentiating $g$, we have $\frac{\partial g_i}{\partial x_j} = \sum_{k = 1}^m w_k \frac{\partial f_k}{\partial x_i \partial x_j}$.
Thus the matrix $\prn{\frac{\partial g_i}{\partial x_j}}_{1 \le i, j \le n}$ is equal to $\sum_{k = 1}^m w_k H(f_k)_x$, where $H(f_k)_x$ is the Hessian matrix of $f_k$ at $x$.
Since $f_k$ is strongly convex, the Hessian matrix $H(f_k)_x$ is positive definite (see \cref{T:criterion storngly convex Hessian}).
Thus, the matrix $\prn{\frac{\partial g_i}{\partial x_j}}_{1 \le i, j \le n}$ is regular on $\Delta^{m - 1}$.
By the implicit function theorem, for any $y \in \Delta^{m - 1}$ there exists an open neighborhood $U_y \subseteq \Delta_{\delta'}^{m - 1}$ and a (unique) $C^{r - 1}$--mapping $x^*_y: U_y \to \R^n$ such that $x^*_y(y) = x^*(y)$ and $g(w, x^*_y(w)) = 0$ for any $w \in U_y$.
We can further deduce from uniqueness of an implicit function that $x^*_y$ coincides with $x^*_{y'}$ on $U_y \cap U_{y'}$ for distinct $y, y' \in \Delta^{m - 1}$.
Since $U = \bigcup_{y \in \Delta^{m - 1}} U_y$ is an open neighborhood of $\Delta^{m - 1}$, one can take $\delta < \delta'$ so that $\Delta^{m - 1}_{\delta}$ is contained in $U$.\footnote{
If not, we can take $x_n \in U^c \cap \Delta^{m - 1}_{1 / n}$ for any $n \in \N$.
Since $\overline{\Delta^{m - 1}_1}$ is compact, $\Set{x_n}_{n \in \N}$ has a cluster point $x$, which is contained in $\Delta^{m - 1}$.
However, $x$ is also contained in $U^c$ since it is closed, contradicting the fact that $\Delta^{m - 1} \varsubsetneq U$.
}
We can define $\widetilde{x^*}: \Delta_{\delta}^{m - 1} \to \R^n$ by $\widetilde{x^*}(w) = x^*_{y} (w)$ for $w \in U_{y}$.
It is easy to see that $\widetilde{x^*}$ is a $C^{r - 1}$--mapping and is an extension of $x^*$.
Thus $x^*$ is a $C^{r - 1}$--mapping.
\end{proof}

\begin{proof}[Proof of 2.A in \cref{T:Pareto_convex}]
For $\varepsilon \ge 0$, we define a subset $D^{m - 1}_\varepsilon \varsubsetneq \R^{m - 1}$ by
\[
D^{m - 1}_\varepsilon = \Set{(z_1, \dots, z_{m - 1}) \in \R^{m - 1} | \sum_i z_i < 1 + \varepsilon, \hspace{.3em}z_i > -\varepsilon }.
\]
We denote the closure $\overline{D_0^{m - 1}}$ by $D^{m - 1}$.
It is easy to check that the projection $p: \Delta_\varepsilon^{m - 1} \to D_\varepsilon^{m - 1}$ defined by $p(w_1, \dots, w_m) = (w_1, \dots, w_{m - 1})$ is a diffeomorphism.
In what follows we will identify $\Delta_\varepsilon^{m - 1}$ with $D_\varepsilon^{m - 1}$ by $p$.

Since $\widetilde{x^*}$ constructed in the proof above is an implicit function of the equation $g(w, x) = 0$, the following equality holds:
\[
0 = \sum_{k = 1}^m w_k (d f_k)_{\widetilde{x^*}(w)} = \sum_{k = 1}^{m - 1} z_k (d f_k)_{\widetilde{x^*}(z)} + \prn{1 - \sum_{k = 1}^{m - 1} z_k} (d f_m)_{\widetilde{x^*}(z)},
\]
where $(z_1, \dots, z_{m - 1}) \in D_{\delta}^{m - 1}$ is a point corresponding to $w \in \Delta^{m - 1}_{\delta}$.
Differentiating the both sides of the equation above by $z_j$ $(j = 1, \dots, m - 1)$, we obtain the following equality:
\begin{align*}
0 = & \sum_{k = 1}^{m - 1} z_k \prn{H(f_k)_{\widetilde{x^*}(z)} \frac{\partial \widetilde{x^*}}{\partial z_j}} + (d f_j)_{\widetilde{x^*}(z)}\\
    & + \prn{1 - \sum_{k = 1}^{m - 1} z_k} \prn{H(f_m)_{\widetilde{x^*}(z)} \frac{\partial \widetilde{x^*}}{\partial z_j}} - (d f_m)_{\widetilde{x^*}(z)}.
\end{align*}
We thus obtain:
\begin{align*}
\frac{\partial \widetilde{x^*}}{\partial z_j}
&= -\prn{\sum_{k = 1}^{m - 1} z_k H(f_k)_{\widetilde{x^*}(z)} + \prn{1 - \sum_{k = 1}^{m - 1} z_k} H(f_m)_{\widetilde{x^*}(z)}}^{-1} \hspace{-.6em} \prn{(d f_j)_{\widetilde{x^*}(z)} - (d f_m)_{\widetilde{x^*}(z)}}\\
&= -A(z) \prn{(d f_j)_{\widetilde{x^*}(z)} - (d f_m)_{\widetilde{x^*}(z)}},
\end{align*}
where we denote the matrix $\prn{\sum_{k = 1}^{m - 1} z_k H(f_k)_{\widetilde{x^*}(z)} + \prn{1 - \sum_{k = 1}^{m - 1} z_k} H(f_m)_{\widetilde{x^*}(z)}}^{-1}$ by $A(z)$, which is positive definite for $z \in D^{m - 1}$.
Since $D^{m - 1}$ is compact and the corank of $d f_{\widetilde{x^*}(z)}$ is equal to $1$ for any $z \in D^{m - 1}$, by retaking $\delta$ if necessary, we can assume that the corank of $d f_{\widetilde{x^*}(z)}$ is equal to $1$ and $A(z)$ is positive definite, and thus regular, for any $z \in D_{\delta}^{m - 1}$.
(Note that the condition $\corank(d f_x) = 1$ is an open condition in $\Crit(f)$.)

We will show that the matrix
\begin{equation}\label{Eq:matrix dfi-dfm}
\begin{pmatrix}
(d f_1)_{\widetilde{x^*}(z)} - (d f_m)_{\widetilde{x^*}(z)} & \cdots & (d f_{m - 1})_{\widetilde{x^*}(z)} - (d f_m)_{\widetilde{x^*}(z)}
\end{pmatrix}
\end{equation}
has rank $m - 1$ for any $z \in D^{m - 1}$.
If not so, there exists $a_i \in \R$ with $(a_1, \dots, a_{m - 1}) \ne 0$ such that
\[
\sum_{i = 1}^{m - 1} a_i \prn{(d f_i)_{\widetilde{x^*}(z)} - (d f_m)_{\widetilde{x^*}(z)}} = 0.
\]
On the other hand, by the definition of $\widetilde{x^*}$, we obtain
\[
\sum_{j = 1}^{m - 1} z_j (d f_j)_{\widetilde{x^*}(z)} + \prn{1 - \sum_{i = 1}^{m - 1} z_i} (d f_m)_{\widetilde{x^*}(z)} = 0.
\]
Thus, two vectors $\prn{z_1, \dots, z_{m - 1}, 1 - \sum_{i = 1}^{m - 1} z_i}$ and $\prn{a_1, \dots, a_{m - 1}, -\sum_{i = 1}^{m - 1} a_i}$ are contained in $\ker d f_{\widetilde{x^*}(z)}$.
However, these are linearly independent and contradict the assumption $\corank \prn{d f_{\widetilde{x^*}(z)}} = 1$.
Therefore, the matrix in \cref{Eq:matrix dfi-dfm} has rank $m - 1$ for any $z \in D^{m - 1}$.
Since the condition that the matrix in \cref{Eq:matrix dfi-dfm} has rank $m - 1$ is an open condition for $z$, we can assume that this condition holds for any $D_{\delta}^{m - 1}$ by making $\delta$ sufficiently small.
The differential
\[
d \widetilde{x^*}_z = -A(z)
\prn{(d f_1)_{\widetilde{x^*}(z)} - (d f_m)_{\widetilde{x^*}(z)}, \dots, (d f_{m - 1})_{\widetilde{x^*}(z)} - (d f_m)_{\widetilde{x^*}(z)}}
\] 
also has rank $m - 1$ since $A(z)$ is regular for any $z \in D_{\delta}^{m - 1}$.

We next show that the mapping $x^*$ is injective.
Assume that $x^*(w)$ is equal to $x^*(w')$ for $w, w' \in \Delta^{m - 1}$.
Since the corank of $d f_{x^*(w)}$ is $1$ and $\sum_{j = 1}^m w_j (d f_j)_{x^*(w)} = 0$, we can obtain $\img(d f_{x^*(w)}) = \gen{w}^{\perp}$.
In the same way, we can also prove that $\img(d f_{x^*(w')})$ is equal to $\gen{w'}^{\perp}$.
From the assumption, we can deduce that $\gen{w}^{\perp}$ is equal to $\gen{w'}^{\perp}$
and thus $w = w'$.

We have shown that $x^*$ is an injective immersion.
Since $\Delta^{m - 1}$ is compact, $x^*$ is a homeomorphism and thus a diffeomorphism to its image, which is equal to $X^*(f)$.
\end{proof}

\begin{proof}[Proof of 2.B in \cref{T:Pareto_convex}]
We first prove that $f|_{X^*(f)}$ is injective.
Let $w, z \in \Delta^{m - 1}$, $x = x^*(w)$ and $y = x^*(z)$.
Suppose that $f(x)$ is equal to $f(y)$.
Then $\prn{\sum_{i = 1}^m w_i f_i}(x)$ is also equal to $\prn{\sum_{i = 1}^m w_i f_i}(y)$.
Since the function $\sum_{i = 1}^m w_i f_i$ is strongly convex, the point minimizing $\sum_{i = 1}^m w_i f_i$ is unique (see \cref{T:minimum}).
Thus, $x$ is equal to $y$.

As we noted in \cref{R:corank_fold}, the corank assumption implies that $n$ is greater than or equal to $m - 1$.
If $n = m - 1$, this assumption further implies that $f$ is an immersion at any point in $X^*(f)$.
The restriction $f|_{X^*(f)}$ is thus an embedding since any injective immersion on a compact manifold is an embedding.
In what follows we will assume $n \ge m $.

We next show that any point $x \in X^*(f)$ is a fold of $f$.
The following transformations preserve strong convexity of $f$:
\begin{itemize}
\item $(f_1, \dots, f_m) \mapsto (f_{\sigma(1)}, \dots, f_{\sigma(m)}) \quad (\sigma \in \mathfrak{S}_m)$,
\item $(f_1, \dots, f_m) \mapsto (f_1, \dots, f_m + \alpha f_i) \quad (\alpha > 0, i = 1, \dots, m - 1)$, and
\item linear transformations of the source of $f$,
\end{itemize}
where $\mathfrak{S}_m$ is the symmetric group of degree $m$.
By applying these transformations if necessary, we can assume the followings:
\[
(d f_m)_x = 0,\quad
\prn{\frac{\partial f_i}{\partial x_j}(x)}_{1 \le i, j \le m - 1} = I_{m - 1},\quad
\frac{\partial f_i}{\partial x_j}(x) = 0 \ 
\begin{pmatrix}
  i = 1, \dots, m\\
  j = m, \dots, n
\end{pmatrix}.
\]
Let $\lambda_f: \R^n \to \R^{n - m + 1}$ be the mapping defined in \cref{S:fold}.
By \cref{T:criterion fold}, it suffices to show the followings:
\begin{enumerate}[(A)]
\item the rank of $(d \lambda_f)_x$ is $n - m + 1$,
\item $\ker (d \lambda_f)_x \oplus \ker d f_x = T_x \R^n$.
\end{enumerate}
By the assumptions above, we can calculate $(d \lambda_f)_x$ as follows:
\begin{align*}
(d \lambda_f)_x
&= \pm \prn{\dfrac{\partial^2 f_m}{\partial x_j \partial x_{m - 1+i}}(x)}_{1 \le i \le n - m + 1, 1 \le j \le n}\\
&= \pm \prn{
\begin{array}{c|c}
\text{\rule{0pt}{15pt} \Large{0}}_{(n - m + 1) \times (m - 1)} & \emph{\Large{I}}_{n - m + 1}
\end{array}
}
H(f_m)_x.
\end{align*}
Since $f_m$ is strongly convex, the Hessian matrix $H(f_m)_x$ is positive definite, in particular regular by \cref{T:criterion storngly convex Hessian}.
Thus, the condition (A) holds.
The above calculation also implies the following equality:
\[
\ker(d \lambda_f)_x = \gen{H(f_m)_x^{-1} e_1, \dots, H(f_m)_x^{-1} e_{m - 1}}.
\]
Let $v \in \ker(d \lambda_f)_x \cap \ker(d f)_x$.
From the equality above, we can find $w \in \R^{m - 1} \times \Set{0} \varsubsetneq \R^m$ such that $v = H(f_m)_x^{-1} w$.
Thus the following holds:
\[
0 = d f_x(v) =
\begin{pmatrix}
  I_{m - 1} & 0\\
  0         & 0
\end{pmatrix}
H(f_m)_x^{-1} w.
\]
We can deduce the following from this equality:
\[
    {}^t w H(f_m)^{-1}_x w = 0.
\]
Since $H(f_m)$ is positive definite by \cref{T:criterion storngly convex Hessian}, $w$ is equal to $0$.
Since the corank of $d f_x$ is equal to $1$, we can deduce the following from the condition (A):
\[
    \dim \ker(d \lambda_f)_x + \dim \ker(d f_x) = n.
\]
Thus the condition (B) also holds.
We can eventually conclude that $f|_{X^*(f)}$ is an immersion.
\end{proof}

%%%%%%%%%%%%%%%%%%%%%%%%%%%%%%%%%%%%%%%%%%%%%%%%%%%%%%%%%%%%%%%%%%%%%%%%%%%%%%%%
\subsection{Examples}\label{S:example}
One of the most simple and representative instances of strongly convex problem is the multiobjective location problem under the Euclidean norm.
It is well known that the Pareto set (resp.\ the Pareto front) of this problem is a convex hull of minimizing points (resp.\ their values) of individual objective functions~\cite{Kuhn1967}.
Thus, if these minimizing points are in general position, then the convex hull becomes a simplex and this problem is a $C^0$--simplicial problem.

In this section we will show that in the strongly convex case, the condition that minimizing points are in general position is no longer necessary nor sufficient to ensure $C^r$--simpliciality, and the corank assumption is still essential to determine the topology of the Pareto set and the Pareto front.
To this end, we will give two examples of strongly convex mappings from $\R^3$ to $\R^3$, and discuss the configurations of Pareto sets of them.
As we mentioned in the beginning of \cref{S:Topology Paretoset}, for any strongly convex mapping $f: \R^3 \to \R^3$ we can define a mapping $x^*: \Delta^2 \to X^*(f)$.
The first example (given in \cref{S:example1}) has a corank $2$ differential at a point in the Pareto set, and the corresponding $x^*$ is not a diffeomorphism (despite the fact that the minimizing points of the three component functions are in general position).
This example implies that we cannot drop the corank assumption in the assertion 2 of \cref{T:Pareto_convex}.
The second example (given in \cref{S:example2}) satisfies the corank assumption, and thus the corresponding $x^*$ is a diffeomorphism (although the minimizing points of the three component functions are \emph{not} in general position).

\begin{example}[general position with corank $2$]\label{S:example1}
Define a mapping $f = (f_1, f_2, f_3): \R^3 \to \R^3$ as follows:
\begin{align*}
    f_1(x, y, z) & = x^2 + y^2 + z^2,\\
    f_2(x, y, z) & = x + y + x^2 + y^2 + z^2,\\
    f_3(x, y, z) & = -(x + y) + x^2 + 2y^2 + z^2.
\end{align*}
The mapping $f$ is strongly convex.
We will check that $X^*(f)$ contains a singularity of corank $2$,
and is not diffeomorphic to $\Delta^2$.
The differentials at $p = (x, y, z) \in \R^3$ are 
\begin{align*}
    d f_{1, p} & = (     2x,      2y, 2z)\\
    d f_{2, p} & = ( 1 + 2x,  1 + 2y, 2z)\\
    d f_{3, p} & = (-1 + 2x, -1 + 4y, 2z),
\end{align*}
and thus the corank of $d f_{\boldsymbol{0}}$ is $2$.
Since $f$ is strongly convex, the mapping $x^*: \Delta^2 \to X^*(f)$ is surjective by the assertion 1 of \cref{T:Pareto_convex}.
Regarding $\Delta^2$ as $D^2 = \Set{(w_2, w_3) | w_2, w_3 \ge 0, w_2 + w_3 \le 1}$,
we obtain
\[
    x^*(w_2, w_3) = \prn{-\frac{w_2 - w_3}{2}, -\frac{w_2 - w_3}{2(1 + w_3)}, 0}.
\]
Obviously $x^*$ maps 
the line defined by $w_2 - w_3 = 0$ in $\Delta^2 = D^2$
into single point (the origin),
while it is injective at points outside the above line.
Thus $X^*(f)( = x^*(\Delta^2))$ is not diffeomorphic to $\Delta^2$.
\Cref{F:ParetoSet-corank2} describes the Pareto set of $f$, together with the contours of the functions $f_1$ (red), $f_2$ (blue) and $f_3$ (green).
\begin{figure}[htbp]
    \centering
    \includegraphics[height=60mm]{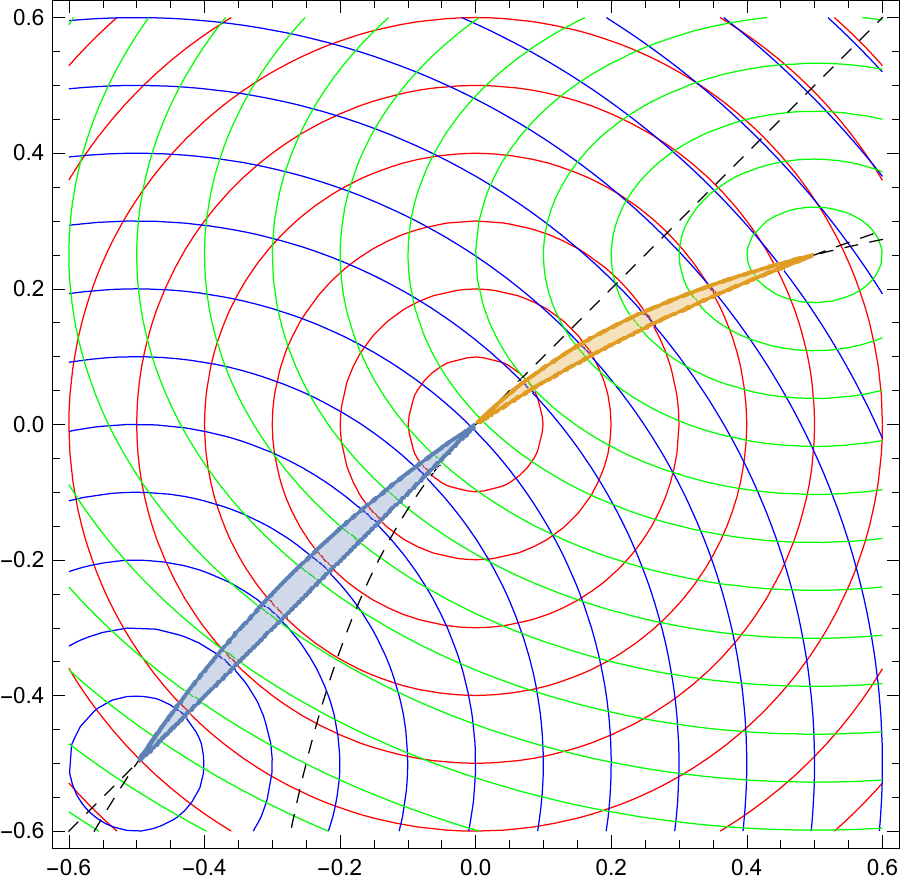}
    \caption{The Pareto set of $f$. The union of two domains colored with orange and blue is the Pareto set.} \label{F:ParetoSet-corank2}
\end{figure}

For $\varepsilon \in \R - \Set{0}$, we define another mapping $h_\varepsilon: \R^3 \to \R^3$ as follows:
\[
    h_\varepsilon (x, y, z) = (f_1(x, y, z) + \varepsilon z, f_2(x, y, z), f_3(x, y, z)).
\]
Note that the mapping $h_\varepsilon$ is a linear perturbation of $f$.
It is easy to verify that the mapping $h_\varepsilon$ is strongly convex and never has corank $2$ critical points.
Thus the problem minimizing $h_\varepsilon$ is simplicial.
We will see in \cref{sec:generic} that in general any strongly convex problem becomes simplicial after a generic linear perturbation (see \cref{T:Pareto_convex generic}).
\end{example}

\begin{example}[non-general position without corank $2$]\label{S:example2}
We define a mapping $f = (f_1, f_2, f_3): \R^3 \to \R^3$ as follows:
\begin{align*}
f_1(x, y, z) & =   x^2      + (-y + x)^2     + z^2,\\
f_2(x, y, z) & = 2(x - 1)^2 + (-y + x - 1)^2 + z^2,\\
f_3(x, y, z) & =  (x - 2)^2 + ( y + x - 2)^2 + z^2.
\end{align*}
The mapping $f$ is strongly convex.
We will check that $f$ satisfies the assumption in the assertion 2 of \cref{T:Pareto_convex} (although the minimizing points of $f_1, f_2, f_3$ are not in general position).
Let $p = (x, y, z)$ be a point in $X^*(f)$.
It is easy to see that $z$ is equal to $0$.
Thus the differentials at $p$ are calculated as follows:
\begin{align*}
\prn{d f_1}_p & = (4x - 2y,     -2x + 2y,     0),\\
\prn{d f_2}_p & = (6x - 2y - 6, -2x + 2y + 2, 0),\\
\prn{d f_3}_p & = (4x + 2y - 8,  2x + 2y - 4, 0).
\end{align*}
Suppose that the corank of $d f_p$ is greater than or equal to $2$.
Then the following equalities hold:
\begin{align}
&0 = \det
\begin{pmatrix}
  4x - 2y     & -2x + 2y\\
  6x - 2y - 6 & -2x + 2y + 2
\end{pmatrix}
= 4 \prn{x - \frac{y}{2} - \frac{1}{2}}^2 - \prn{y - 3}^2 + 8,\label{Eq:hyperbola1}\\
&0 = \det
\begin{pmatrix}
  4x - 2y     & -2x + 2y \\
  4x + 2y - 8 &  2x + 2y - 4
\end{pmatrix}
= 2 \prn{8 \prn{x - 1}^2 - 4 \prn{y - \frac{3}{2}}^2 + 1}.\label{Eq:hyperbola2}
\end{align}
These equalities give rise to two hyperbolas given in \cref{F:hyperbolas} (the red hyperbola is defined by \cref{Eq:hyperbola1}, while the blue one is defined by \cref{Eq:hyperbola2}).
As shown in \cref{F:hyperbolas}, the two hyperbolas intersect at two points.
One is the origin $\boldsymbol{0}$ and let $q = (x', y')$ be the other.
Since the rank of $\prn{(d f_1)_{\boldsymbol{0}},(d f_2)_{\boldsymbol{0}},(d f_3)_{\boldsymbol{0}}}$ is $2$, $(x, y)$ is not equal to $(0, 0)$.
Thus we obtain $(x, y) = (x', y')$.
However, since $y' - x' > 1$ and $x', y' > 1$, all of the three values $-2x' + 2y', -2x' + 2y' + 2$ and $2x' + 2y' - 4$ are greater than $0$, contradicting \cref{T:chara-pareto}.
Hence we can conclude that there is no point $p \in X^*(f)$ with $\corank(d f_p) \ge 2$.
\Cref{F:ParetoSet} describes the Pareto set of $f$, together with the contours of the functions $f_1$ (red), $f_2$ (blue) and $f_3$ (green).
\begin{figure}[htbp]
\centering
\subfigure[Two hyperbolas.]{\includegraphics[height=60mm]{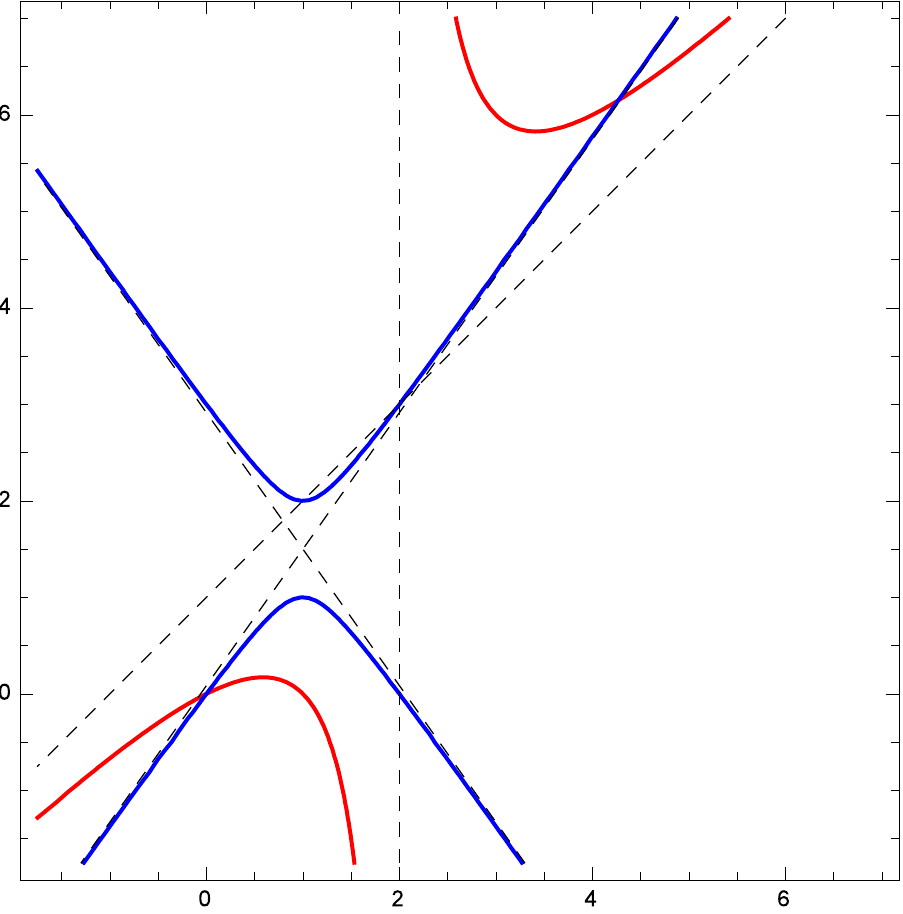}\label{F:hyperbolas}}
\subfigure[The Pareto set of $f$ (projected on the $xy$--plane).]{\includegraphics[height=60mm]{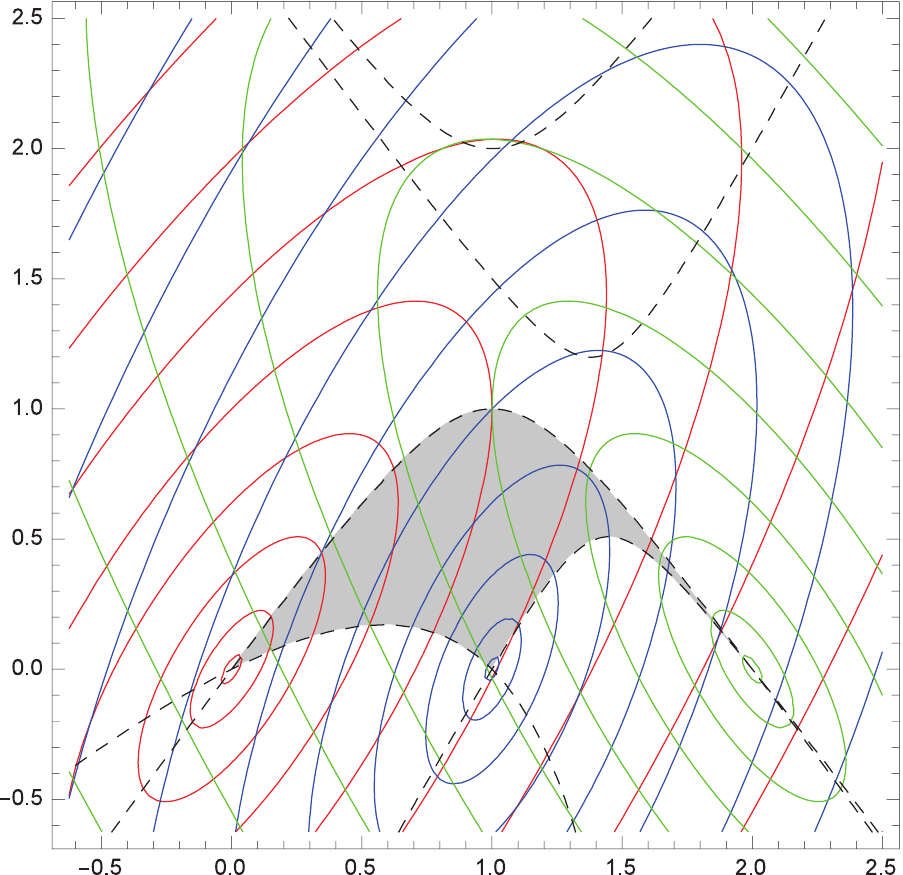}\label{F:ParetoSet}}
\caption{The dotted lines in \cref{F:hyperbolas} are asymptotes of the hyperbolas.} 
\label{F:MathematicaGraph1}
\end{figure}
\end{example}

%%%%%%%%%%%%%%%%%%%%%%%%%%%%%%%%%%%%%%%%%%%%%%%%%%%%%%%%%%%%%%%%%%%%%%%%%%%%%%%%
\section{Generic linear perturbations of strongly convex mappings}\label{sec:generic}
%%%%%%%%%%%%%%%%%%%%%%%%%%%%%%%%%%%%%%%%%%%%%%%%%%%%%%%%%%%%%%%%%%%%%%%%%%%%%%%%
In this section, we will investigate the multiobjective optimization problem minimizing a generically linearly perturbed strongly convex mapping.
Let $\mathcal{L}(\R^n, \R^m)$ be the space consisting of all linear mappings from $\R^n$ into $\R^m$.
In what follows we will regard $\mathcal{L}(\R^n, \R^m)$ as the Euclidean space $(\R^n)^m$ in the obvious way.
The purpose of this section is to show the following theorem:
\begin{theorem}\label{T:Pareto_convex generic}
Let $f: \R^n \to \R^m$ $(n \ge m)$ be a strongly convex $C^r$--mapping $(2 \le r \le \infty)$.
If $n - 2m + 4 > 0$, then there exists a subset $\Sigma \varsubsetneq \mathcal{L}(\R^n, \R^m)$ with Lebesgue measure zero such that for any $\pi \in \mathcal{L}(\R^n, \R^m) - \Sigma$ the mapping $f + \pi$ never has differential with corank greater than $1$ on its Pareto set.
In particular, the multiobjective optimization problem minimizing $f + \pi: \R^n \to \R^m$ is $C^{r - 1}$--simplicial.
\end{theorem}

We begin with observing that strong convexity is preserved under linear perturbations.

\begin{lemma}\label{T:generic_multi}
Let $f: \R^n \to \R^m$ be a strongly convex mapping.
Then, for any $\pi \in \mathcal{L}(\R^n, \R^m)$, the mapping $f + \pi: \R^n \to \R^m$ is also a strongly convex mapping.
\end{lemma}

\begin{proof}[Proof of \cref{T:generic_multi}]
Obviously, it is sufficient to show the statement under the assumption that $f$ is a function (i.e.~$m = 1$).
For $x, y \in \R^n$ and $t \in [0, 1]$, the following holds:
\begin{align*}
&  t \prn{(f + \pi)(x)} + (1-t)\prn{(f + \pi)(y)} - (f + \pi)(tx + (1 - t)y)\\ 
&= t \prn{f(x) + \pi(x)} + (1-t)\prn{f(y) + \pi(y)} - f(tx + (1 - t)y) - \pi(tx + (1 - t)y)\\ 
&= tf(x) + (1 - t)f(y) - f(tx + (1 - t)y),
\end{align*}
where the last equality holds since $\pi$ is linear.
Since $f$ is strongly convex, there exists $\alpha > 0$ satisfying the following inequality for any $x, y \in \R^n$ and $t \in [0, 1]$:
\[
    tf(x) + (1 - t)f(y) - f(tx + (1 - t)y) \ge \frac{1}{2} \alpha t(1 - t) \norm{x - y}^2.
\]
Hence, the mapping $f + \pi$ is also strongly convex.
\end{proof}

Before proving \cref{T:Pareto_convex generic}, we will briefly review the result in \cite{Ichiki-generic} needed here.
Let $S_k \varsubsetneq J^1(\R^n, \R^m)$ be the subset defined in \cref{S:fold}.
It is known that $S_k$ is a submanifold of $J^1(\R^n, \R^m)$ satisfying the following (see \cite{Golubitsky1974}):
\[
    \codim S_k = (n - v + k)(m - v + k)
\]
where $\codim S_k = \dim J^1(\R^n, \R^m) - \dim S_k$ and $v = \min \Set{n, m}$.
The following lemma is merely a special case of \cite[Theorem 1]{Ichiki-generic}:

\begin{lemma}[cf.~\cite{Ichiki-generic}]\label{T:generic_transversality}
Let $f: \R^n \to \R^m$ be a $C^r$--mapping.
Let $k$ be an integer satisfying $1 \le k \le \min \Set{n, m}$.
If 
$r > \max \Set{n - \codim S_k, 0} + 1$, 
then there exists a subset $\Sigma \varsubsetneq \mathcal{L}(\R^n, \R^m)$ with Lebesgue measure zero such that 
for any $\pi \in \mathcal{L}(\R^n, \R^m) - \Sigma$, the mapping 
$j^1(f + \pi): \R^n \to J^1(\R^n, \R^m)$ is transverse to $S_k$.
\end{lemma}

\begin{proof}[Proof of \cref{T:Pareto_convex generic}]
In the case $m = 1$, it is clearly seen that \cref{T:Pareto_convex generic} holds.
Hence, we will consider the case $m \ge 2$.
Since $n \ge m$, the codimension of $S_2$ is equal to $2(n - m + 2)$.
By the assumption $n - 2m + 4 > 0$, we can obtain the inequality $\codim S_2 > n$.
Let $k$ be an integer with $2 \le k \le m$.
It follows that 
\[
    n - \codim S_k \le n - \codim S_2 < 0.
\]
In particular, for a mapping $g: \R^n \to \R^m$, transversality of $j^1 g$ to $S_k$ is equivalent to the condition that $j^1 g(\R^n) \cap S_k = \emptyset$, that is, $g$ has no corank $k$ critical points (see \cite[Ch.~II, Proposition 4.2]{Golubitsky1974}).
Furthermore, the following inequality holds:
\[
    r \ge 2 > \max \Set{n - \codim S_k, 0} + 1.
\]
We can deduce from \cref{T:generic_transversality}, together with the observations above, that there exists $\Sigma_k \varsubsetneq \mathcal{L}(\R^n, \R^m)$ with Lebesgue measure zero such that the mapping 
$f + \pi$ has no corank $k$ critical points for any $\pi \in \mathcal{L}(\R^n, \R^m) - \Sigma_k$.
Set $\Sigma = \bigcup_{l = 2}^m \Sigma_l \varsubsetneq \mathcal{L}(\R^n, \R^m)$, which also has Lebesgue measure zero.
Lastly, we can easily verify that $\Sigma$ satisfies the conditions in \cref{T:Pareto_convex generic}.
\end{proof}

\begin{remark}\label{R:assumption dimensions}
We cannot drop the assumption $n-2m+4>0$ in \cref{T:Pareto_convex generic}. 
Let
$G(x) = \prn{g_1(x) - x_3, g_2(x) - x_4, g_1(x) + x_3, g_2(x) + x_4}$ be a map from 
\begin{math}
\R^4
\end{math}
to
\begin{math}
\R^4
\end{math}, where
\begin{align*}
    g_1(x) &= x_1^2 + x_3 x_2 + \frac{1}{2} \prn{x_2^2 + x_4 x_1} + x_3^2 + x_4^2, \\
    g_2(x) &= x_2^2 + x_4 x_1 + \frac{1}{2} \prn{x_1^2 + x_3 x_2} + x_3^2 + x_4^2. \label{eq_corank2}
\end{align*}
The mapping $G$ is strongly convex and has a corank $2$ critical point at the origin. 
In what follows, we will see that the Pareto set of any small linear perturbation of $G$ contains a corank $2$ critical point. 
(Note that the inequality $n - 2m + 4 > 0$ does not hold when $n = m = 4$.)

The Jacobi matrix of $G$ at the origin is 
$d G_0 =
\prn{\begin{smallmatrix}
O & O \\
-I  &I
\end{smallmatrix}}
$, where $I$ and $O$ are $2 \times 2$ unit matrix and zero matrix, respectively.
Define the cokernel of $dG_0$ as
\begin{equation*}
\coker dG_0 = \Set{\prn{v_1,v_2,v_3,v_4} \in \R^4 | \sum_{i=1}^4 v_i \prn{dG_i}_0 = 0},
\end{equation*}
then, the subspace $\coker dG_0$ is equal to $\gen{(1, 0, 1, 0),\ (0, 1, 0, 1)}$. 
In particular, the subspace $\coker dG_0$ intersects with the interior of $\Delta^3$.  

For $(x, \pi) \in \R^4 \times \mathcal{L}(\R^4, \R^4)$, put $d(G + \pi)_x =
\prn{\begin{smallmatrix}
A(x, \pi) & B(x, \pi) \\
C(x, \pi) & D(x, \pi)
\end{smallmatrix}}$
where $A(x, \pi)$, $B(x, \pi)$, $C(x, \pi)$, and $D(x, \pi)$ are $2 \times 2$ matrices.
We can take an open neighborhood $U$ of $(0, 0) \in \R^4 \times \mathcal{L}(\R^4, \R^4)$ so that the matrix $D(x, \pi)$ is invertible for any $(x, \pi) \in U$.
By multiplying the matrix
$\prn{\begin{smallmatrix}
I & O \\
-D^{-1}(x, \pi) C(x, \pi) & I
\end{smallmatrix}}$
to $d(G + \pi)_x$ from the right, we obtain the matrix
$\prn{\begin{smallmatrix}
A(x, \pi) - B(x, \pi) D^{-1}(x, \pi) C(x, \pi) & B(x, \pi) \\
O & D(x, \pi)
\end{smallmatrix}}$.
Put $E(x, \pi) = A(x, \pi) - B(x, \pi) D^{-1}(x, \pi) C(x, \pi)$.
The matrix $E(0, 0)$ is the zero matrix, and
\begin{align*}
\left. \frac{\partial E(x, \pi)}{\partial x_1} \right|_{(x, \pi) = (0, 0)} =
\begin{pmatrix}
4 & 2 \\
0 & 0
\end{pmatrix},&\quad
\left. \frac{\partial E(x, \pi)}{\partial x_2} \right|_{(x, \pi) = (0, 0)} =
\begin{pmatrix}
0 & 0 \\
2 & 4
\end{pmatrix}, \\
\left. \frac{\partial E(x, \pi)}{\partial x_3} \right|_{(x, \pi) = (0, 0)} =
\begin{pmatrix}
0 & 0 \\
2 & 1
\end{pmatrix},&\quad
\left. \frac{\partial E(x, \pi)}{\partial x_4} \right|_{(x, \pi) = (0, 0)} =
\begin{pmatrix}
1 & 2 \\
0 & 0
\end{pmatrix}
\end{align*}
hold.
These equalities imply that the rank of the Jacobi matrix of $E$ with respect to $x$ is $4$. 
Therefore, applying the implicit function theorem we can obtain $\hat{x}: V \to \R^4$, where $V$ is an open neighborhood of $0 \in \mathcal{L}(\R^4, \R^4)$, such that the equality $E \prn{\hat{x}(\pi), \pi} = O$ holds for any $\pi \in V$.
Note that $\hat{x}(\pi)$ is continuous with respect to $\pi$.
Then, we can define a continuous mapping $\coker d(G + \cdot)_{\hat{x}(\cdot)}: V \to Gr (2, \R^4)$ where $Gr(2, \R^4)$ is all $2$-dimensional linear subspaces of $\R^4$.
Since $\coker dG_0$ intersects with the interior of $\Delta^3$, so does $\coker d(G + \pi)_{\hat{x}(\pi)}$ if $\pi$ is sufficiently small.
This proves that $\hat{x}(\pi)$ is a Pareto solution to $G + \pi$ with corank $2$ differential for a sufficiently small $\pi$.
\end{remark}

\begin{remark}\label{rem:Gamma}
Let $f = (f_1, \dots, f_m): \R^n \to \R^m$ be a strongly convex $C^r$--mapping $(2 \le r \le \infty)$.
We can deduce from \cref{T:generic_multi} that the mapping $\sum_{i = 1}^m w_i (f_i + \pi_i)$ is strongly convex for any $(w, \pi) \in \Delta^{m - 1}\times \mathcal{L}(\R^n, \R^m)$, where $\pi = (\pi_1, \dots, \pi_m)$.
By the same argument as in \cref{S:Topology Paretoset}, we can define a mapping $\Gamma: \Delta^{m - 1} \times \mathcal{L}(\R^n, \R^m) \to \R^n$ as follows:
\[
    \Gamma(w, \pi) = \arg\min_{x \in \R^n} \prn{\sum_{i = 1}^m w_i (f_i + \pi_i)(x)}. 
\]
Then the mapping $\Gamma$ is continuous, and thus, for any $\varepsilon > 0$ there exists an open neighborhood $U$ of $0 \in \mathcal{L}(\R^n, \R^m)$ satisfying the following inequality for any $(w, \pi) \in \Delta^{m - 1} \times U$:
\[
    \norm{\Gamma(w, \pi) - \Gamma(w, 0)} < \varepsilon.
\]
This inequality implies that the Pareto set of a linearly perturbed mapping is close to that of the original one. 
For the proof of the statement here, see \cref{A:the proof for remark}.
\end{remark}

%%%%%%%%%%%%%%%%%%%%%%%%%%%%%%%%%%%%%%%%%%%%%%%%%%%%%%%%%%%%%%%%%%%%%%%%%%%%%%%%
\section{Applications}\label{sec:applications}
%%%%%%%%%%%%%%%%%%%%%%%%%%%%%%%%%%%%%%%%%%%%%%%%%%%%%%%%%%%%%%%%%%%%%%%%%%%%%%%%
As we have seen, strongly convex problems have a variety of desirable properties which make their Pareto sets easy to understand.
Although lots of practical problems are \emph{not} necessarily strongly convex, we can apply suitable ``structure-preserving'' transformations to some of these problems so that they become strongly convex.
In this section, we will give several examples of such problems.

%%%%%%%%%%%%%%%%%%%%%%%%%%%%%%%%%%%%%%%%%%%%%%%%%%%%%%%%%%%%%%%%%%%%%%%%%%%%%%%%
\subsection{Location problems}\label{sec:flp}
One of the most traditional examples is the location problem, which requires to find the best place $x \in \R^n$ for a facility so that the weighted sum $\sum_{i = 1}^m w_i \norm{x - p_i}$ (for given $w \in \Delta^{m - 1}$) of distances from demand points $p_1, \dots, p_m \in \R^n$ is minimized.
Its multiobjective version~\cite{Kuhn1967} is the following problem:\footnote{While Kuhn~\cite{Kuhn1967} originally considered a planar case ($n=2$), we consider the problem in general dimension.}
\begin{equation}\label{eqn:flp-mop}
\begin{split}
\text{minimize } & f(x)   = (f_1(x), \dots, f_m(x)) \text{ subject to } x \in \R^n\\
\text{where }    & f_i(x) = \norm{x - p_i} \quad(i = 1, \dots, m).
\end{split}
\end{equation}
The mapping $f$ is called a \emph{distance mapping}~\cite{Ichiki2013}, which is not differentiable.
Each $f_i$ is convex but not strongly convex, and thus so is the problem~\cref{eqn:flp-mop}.

Let us consider the transformation of the target $T: [0, \infty)^m \to [0, \infty)^m$ defined by $T(y_1, \dots, y_m) =(y_1^2, \dots, y_m^2)$, which preserves the Pareto ordering of $[0, \infty)^m$.
We have a transformed problem
\begin{equation}\label{eqn:flp-scmop}
\text{minimize } T \circ f(x) \text{ subject to } x \in \R^n.
\end{equation}
The mapping $T \circ f$ (called a \emph{distance-squared mapping}~\cite{Ichiki2013}) is differentiable and strongly convex, in particular the problem \cref{eqn:flp-scmop} is strongly convex.
Since $T$ preserves the Pareto ordering, the Pareto sets of \cref{eqn:flp-mop,eqn:flp-scmop} are identical and the Pareto fronts are homeomorphic.

For the original problem \cref{eqn:flp-mop}, there is a weight with which the weighted sum scalarization has non-unique solutions (e.g.~the problem minimizing $\sum_{i = 1}^m w_i \norm{x - p_i}$ with $w_1 = w_2 = 1/2$, $w_3 = \dots = w_m = 0$ has solutions $t p_1 + (1 - t) p_2$ for $t \in [0, 1]$), in particular we cannot define a mapping $x^*$ given in \cref{S:Topology Paretoset}.
On the other hand, for the transformed problem \cref{eqn:flp-scmop}, every scalarized problem has a unique solution and the entire Pareto set consists of such elements by the assertion 1 of \cref{T:Pareto_convex}.
It is further easy to verify that the corank of $d(T \circ f)_x$ is $1$ for any $x \in X^*(T \circ f)$, provided that $n \ge m - 1$ and $p_1, \dots, p_m$ are in general position.
Thus, the assertion 2 of \cref{T:Pareto_convex} guarantees that the problem \cref{eqn:flp-scmop} is ($C^\infty$--)simplicial.
Since $T$ preserves the Pareto ordering, one can easily see that the problem \cref{eqn:flp-mop} is also ($C^0$--)simplicial.

%%%%%%%%%%%%%%%%%%%%%%%%%%%%%%%%%%%%%%%%%%%%%%%%%%%%%%%%%%%%%%%%%%%%%%%%%%%%%%%%
\subsection{Phenotypic divergence model}\label{sec:pdm}
Another example minimizing distances from points arises in evolutionary biology.
Let $A_i$ be a symmetric, positive definite matrix of size $n$ and $p_i \in \R^n$ ($i = 1, \dots, m$).
Shoval~et~al.~\cite{Shoval2012} provided a model for describing phenotypic divergence of species, which is an extension of the location problem:
\begin{equation}\label{eqn:pdm-mop}
\begin{split}
\text{minimize } f(x)   &= (f_1(x), \dots, f_m(x)) \text{ subject to } x \in \R^n\\
\text{where }    f_i(x) &= \norm{A_i (x - p_i)} \quad (i = 1, \dots, m)
\end{split}
\end{equation}
As before, the problem minimizing \cref{eqn:pdm-mop} is convex but not strongly convex.
We can again apply the transformation $T$ used in the previous subsection and obtain
\begin{equation}\label{eqn:pdm-scmop}
\text{minimize } T \circ f(x) \text{ subject to } x \in \R^n.
\end{equation}
Since affine transformations of the source space preserve strong convexity of a problem, each component of $T \circ f$ (and thus the problem \cref{eqn:pdm-scmop}) is strongly convex.
Applying the assertion 1 of \cref{T:Pareto_convex} we can conclude that both of the problems \cref{eqn:pdm-mop,eqn:pdm-scmop} are weakly simplicial.
In order to further show that these problems are simplicial, we have to check the corank condition in the assertion 2 of \cref{T:Pareto_convex}, which would be a hard task, even if the demand points $p_1, \dots, p_m$ are in general position.
Indeed, problems appearing in \cref{S:example} are special cases of the problems we are dealing with here.
As discussed in \cref{S:example}, generality of the configuration of demand points does not necessarily imply the corank condition.

%%%%%%%%%%%%%%%%%%%%%%%%%%%%%%%%%%%%%%%%%%%%%%%%%%%%%%%%%%%%%%%%%%%%%%%%%%%%%%%%
\subsection{Ridge regression}
The ridge regression~\cite{Hoerl1970} can be reformulated as a multiobjective strongly convex problem.
Let us consider a linear regression model
\[
    y = \theta_1 x_1 + \theta_2 x_2 + \dots + \theta_p x_p + \varepsilon
\]
where $x_1, \dots, x_p$ are predictor variables, $y$ is a response variable, $\varepsilon$ is a Gaussian random variable expressing noise, and $\theta_1, \dots, \theta_p$ are the predictors' coefficients to be estimated from observations.
Given $n$ observations of $x_1, \dots, x_p$ and $y$, which are denoted by $\overline{x}_{i1},\ldots, \overline{x}_{ip}$ and $\overline{y}_i$ ($i=1,\ldots,n$), an observation matrix and a response vector are formed as
\[
    \overline{X} = \begin{pmatrix}
        \overline{x}_{11} & \dots & \overline{x}_{1p}\\
        \vdots & \ddots & \vdots\\
        \overline{x}_{n1} & \dots & \overline{x}_{np}
    \end{pmatrix},\quad
    \overline{y} = \begin{pmatrix}
    \overline{y}_1\\
    \vdots\\
    \overline{y}_n
    \end{pmatrix}.
\]
The ridge regressor is the solution to the following problem:
\begin{equation}\label{eqn:ridge}
\text{minimize } g^\lambda(\theta) = \norm{\overline{X} \theta - \overline{y}}^2 + \lambda \norm{\theta}^2 \text{ subject to } \theta \in \R^p
\end{equation}
where $\norm{\cdot}$ is the Euclidean norm and $\lambda$ is a positive number predetermined by users.
To obtain a good regressor, users have to find an appropriate value of $\lambda$ by repeatedly solving the problem \cref{eqn:ridge} with various candidates of $\lambda$.

We consider the following multiobjective reformulation of the \cref{eqn:ridge}:
\begin{equation}\label{eqn:ridge-mop}
\begin{split}
    \text{minimize } & f^\mu(\theta)   = (f_1^\mu(\theta), f_2(\theta)) \text{ subject to } \theta \in \R^p\\
    \text{where }    & f_1^\mu(\theta) = \norm{\overline{X}\theta - \overline{y}}^2 + \mu \norm{\theta}^2 \quad (\mu > 0),\\                  & f_2(\theta) = \norm{\theta}^2.
\end{split}
\end{equation}
Notice that $\norm{\overline{X}\theta - \overline{y}}^2$ is convex but not ensured to be strongly convex.
We thus add $\mu \norm{\theta}^2$ to guarantee strong convexity of $f_1^\mu$.
By \cref{T:main theorem,T:Pareto_convex}, this problem is weakly simplicial and the mapping
\begin{equation}\label{eqn:ridge-sop}
    \theta^*(w) = \arg\min_\theta \prn{w_1 f_1^\mu(\theta) + w_2 f_2(\theta)}
\end{equation}
is well-defined and continuous on $\Delta^1$, satisfying $\theta^*(\Delta_I) = X^*(f^\mu_I)$ for all $I \subseteq \set{1, 2}$.

Note that for $w=(w_1,w_2)\in \Delta^1- \set{(0,1)}$, the point $\theta^*(w)$ is the minimizer of the function $g^{\lambda(w)}$, where $\lambda(w)= \mu + \frac{w_2}{w_1}$. 
In particular, we can obtain the solutions to the original problems \cref{eqn:ridge} for any hyper-parameter $\lambda\geq \mu$ by solving the multiobjective problem \cref{eqn:ridge-mop}.
Since the problem \cref{eqn:ridge-mop} is weakly simplicial, the mapping $\theta^*: \Delta^1 \to \R^p$ in \cref{eqn:ridge-sop} can be approximated by a B\'ezier simplex with small samples (see~\cite[Section 4]{Tanaka2019}).

\begin{remark}
As the ridge regression problem is a special case of $\ell_p$--regularized regression problems, it is quite natural to expect that one can apply the same idea to solve various sparse modeling problems, including the lasso \cite{Tibshirani1996}, the group lasso~\cite{Yuan2006}, the fused lasso \cite{Tibshirani2005}, the smooth lasso \cite{Hebiri2011}, and the elastic net \cite{Zou2005}.
For example, the group lasso~\cite{Yuan2006}
\begin{equation*}\label{eqn:group-lasso}
\text{minimize } \norm{\overline{X} \theta - \overline{y}}^2 + \lambda \sum_{i = 1}^m \norm{\theta_{(i)}} \text{ subject to } \theta \in \R^p
\end{equation*}
is reformulated as
\begin{equation*}\label{eqn:group-lasso-mop}
\begin{split}
    \text{minimize } & f(\theta)   = (f_0(\theta), \dots, f_m(\theta)) \text{ subject to } \theta \in \R^p\\
    \text{where }    & f_0(\theta) = \norm{\overline{X}\theta - \overline{y}}^2,\\
                     & f_i(\theta) = \norm{\theta_{(i)}}^2 \quad (i = 1, \dots, m).
\end{split}
\end{equation*}
Here, $\theta_{(i)}$ is a coefficient vector corresponding to the $i$-th group of variables.
Each group-wise regularization term is squared to be a $C^2$--function without changing the Pareto ordering.
Note that the original formulation uses the single regularization constant $\lambda$ for all the groups but our reformulation makes us possible to use different regularization constants for different groups and investigate their consequences.
Unfortunately, the functions $f_0,\ldots,f_m$ appearing in the reformulated problem are not always strongly convex.
In order to apply our technique, we have to take other strongly convex functions $\tilde{f}_0,\ldots,\tilde{f}_m$ approximating to the original ones (say, $\tilde{f}_i = f_i +\mu \norm{\theta}^2$ for small $\mu >0$) so that the resulting Pareto set and front are close to those for the original problem.
The issue arising here will be addressed in a forthcoming project.
\end{remark}

%%%%%%%%%%%%%%%%%%%%%%%%%%%%%%%%%%%%%%%%%%%%%%%%%%%%%%%%%%%%%%%%%%%%%%%%%%%%%%%%
\section{Conclusions}\label{sec:conclusions}
%%%%%%%%%%%%%%%%%%%%%%%%%%%%%%%%%%%%%%%%%%%%%%%%%%%%%%%%%%%%%%%%%%%%%%%%%%%%%%%%
In this paper, we have shown that $C^r$--strongly convex problems are $C^{r - 1}$--weakly simplicial.
We have further proved that they are $C^{r - 1}$--simplicial under some mild assumption on the coranks of the differentials of the objective mappings.
The example given after the proof has illustrated the necessity of the corank assumption.

We have also shown that one can always make any strongly convex problem satisfy the corank assumption by a generic linear perturbation, provided that the dimension of the decision space is sufficiently larger than that of the objective space.
Note that this theorem would not hold without the assumption on the dimension pair.
We have proved that a small linear perturbation does not change the Pareto set considerably.

While lots of multiobjective optimization problems appearing in practice are not strongly convex, we have demonstrated that several examples of such problems can be reduced to strongly convex problems via transformations preserving the Pareto ordering and the topology.
The location problems, a phenotypic divergence model, and the ridge regression have been reformulated into multiobjective strongly convex problems and shown to be (weakly) simplicial.

We plan to extend the theorems to those for $C^1$--mappings.
To do this, we will require different techniques since one cannot define the Hessian matrices for $C^1$--mappings.
Another interesting research project is to find a transformation that makes problems strongly convex without causing substantial changes of the Pareto set.

%%%%%%%%%%%%%%%%%%%%%%%%%%%%%%%%%%%%%%%%%%%%%%%%%%%%%%%%%%%%%%%%%%%%%%%%%%%%%%%%
\section*{Acknowledgements}
%%%%%%%%%%%%%%%%%%%%%%%%%%%%%%%%%%%%%%%%%%%%%%%%%%%%%%%%%%%%%%%%%%%%%%%%%%%%%%%%
K.~H. was supported by JSPS KAKENHI Grant Number JP17K14194.
S.~I. was supported by JSPS KAKENHI Grant Number JP19J00650.
Y.~K. was supported by JSPS KAKENHI Grant Number JP16J02200.
H.~T. was supported by JSPS KAKENHI Grant Number JP19K03484 and JST PRESTO Grant Number JPMJPR16E8.
The authors were supported by RIKEN Center for Advanced Intelligence Project.
This work is based on the discussions at 2017 IMI Joint Use Research Program, Short-term Joint Research ``Visualization of solution sets in many-objective optimization via vector-valued stratified Morse theory'' and was supported by the Research Institute for Mathematical Sciences, a Joint Usage/Research Center located in Kyoto University.

%%%%%%%%%%%%%%%%%%%%%%%%%%%%%%%%%%%%%%%%%%%%%%%%%%%%%%%%%%%%%%%%%%%%%%%%%%%%%%%%
\bibliographystyle{plain}
\bibliography{references}
%%%%%%%%%%%%%%%%%%%%%%%%%%%%%%%%%%%%%%%%%%%%%%%%%%%%%%%%%%%%%%%%%%%%%%%%%%%%%%%%
\appendix
\section{The effect of linear perturbations on the Pareto sets}\label{A:the proof for remark}
In this appendix we will show the following statement mentioned in \cref{rem:Gamma}: 

%%%%%%%%%%%%%%%%%%%%%%
\begin{proposition}\label{T:para generic}
The mapping $\Gamma$ defined in \cref{rem:Gamma} is continuous. 
Moreover, for any $\varepsilon > 0$, there exists an open neighborhood $U$ of $0 \in \mathcal{L}(\R^n, \R^m)$ satisfying the following inequality for any $(w, \pi)\in \Delta^{m - 1} \times U$:
\[
    \norm{\Gamma(w, \pi) - \Gamma(w, 0)} < \varepsilon.
\]
\end{proposition}
%%%%%%%%%%%%%%%%%%%%%%%%
\begin{proof}[Proof of \cref{T:para generic}]
Let $(\widetilde{w}, \widetilde{\pi}) \in \Delta^{m - 1} \times \mathcal{L}(\R^n, \R^m)$ be an arbitrary element.
We will show that $\Gamma$ is continuous at the point $(\widetilde{w}, \widetilde{\pi})$. 

Now, let $\gamma: \R^{m - 1} \times \mathcal{L}(\R^n, \R^m) \times \R^n \to \R^n$ be the mapping defined by
\[
    \gamma(w_1, \ldots, w_{m - 1}, \pi, x) = d\prn{\sum_{i = 1}^{m - 1} w_i (f_i + \pi_i) + \prn{1 - \sum_{i = 1}^{m - 1} w_i}(f_m + \pi_m)}_x.
\]
Then, $\gamma$ is a $C^{r-1}$--mapping. 
Set $\widetilde{x} = \Gamma(\widetilde{w}, \widetilde{\pi})$. 
By the definition of $\Gamma$, we have
\[
    d\prn{\sum_{i = 1}^m \widetilde{w}_i (f_i + \widetilde{\pi}_i)}_{\widetilde{x}} = 0,
\]
where $\widetilde{w} = (\widetilde{w}_1, \ldots, \widetilde{w}_m)$ and $\widetilde{\pi} = (\widetilde{\pi}_1, \ldots, \widetilde{\pi}_m)$.
Since $\widetilde{w}_m = 1 - \sum_{i = 1}^{m - 1} \widetilde{w}_i$, we get $\gamma\prn{p(\widetilde{w}), \widetilde{\pi}, \widetilde{x}} = 0$, where $p: \Delta^{m - 1} \to \R^{m - 1}$ is the mapping defined by
\[
    p(w_1, \ldots, w_m) = (w_1, \ldots, w_{m - 1}).
\]

Let $\gamma_{(p(\widetilde{w}), \widetilde{\pi})}: \R^n \to \R^n$ be the mapping defined by $\gamma_{(p(\widetilde{w}), \widetilde{\pi})}(x) = \gamma(p(\widetilde{w}), \widetilde{\pi}, x)$.
It is easy to verify that the following equality holds:
\[
    d(\gamma_{(p(\widetilde{w}), \widetilde{\pi})})_{\widetilde{x}} = H \prn{\sum_{i = 1}^m \widetilde{w}_i (f_i + \widetilde{\pi}_i)}_{\widetilde{x}},
\]
where $H \prn{\sum_{i = 1}^m \widetilde{w}_i (f_i + \widetilde{\pi}_i)}_{\widetilde{x}}$ is the Hessian matrix of $\sum_{i = 1}^m \widetilde{w}_i (f_i + \widetilde{\pi}_i)$ at $\widetilde{x}$.
Since the function $\sum_{i = 1}^m \widetilde{w}_i (f_i + \widetilde{\pi}_i)$ is strongly convex, we can deduce from \cref{T:criterion storngly convex Hessian} that the determinant $\det d(\gamma_{(p(\widetilde{w}), \widetilde{\pi})})_{\widetilde{x}}$ is not equal to $0$.
Therefore, by the implicit function theorem, there exist an open neighborhood $W$ of $(p(\widetilde{w}), \widetilde{\pi}) \in \R^{m - 1} \times \mathcal{L}(\R^n, \R^m)$ and 
a mapping $\varphi: W \to \R^n$ satisfying
\[
    \gamma(z, \pi, \varphi(z,\pi)) = 0
\]
for any $(z, \pi) \in W \subseteq \R^{m - 1} \times \mathcal{L}(\R^n, \R^m)$.
Thus the mapping $\Gamma$ is continuous at $(\widetilde{w},\widetilde{\pi})$ since $\Gamma(w, \pi)$ is equal to $\varphi(p(w), \pi)$ for any $(w, \pi) \in (p \times \mathrm{id})^{-1}(W)$.

Let $\varepsilon > 0$ be an arbitrary real number. 
Since $\Gamma$ is continuous and $\Delta^{m-1}$ is compact, 
there exist an open covering $\set{V_i}_{i=1}^l$ of $\Delta^{m-1}$ 
and open neighborhoods 
$U_1, \ldots, U_l$ of 
$0 \in \mathcal{L}(\R^n,\R^m)$ 
satisfying 
\[
    \norm{\Gamma(w, \pi) - \Gamma(w, 0)} < \varepsilon
\]
for any $(w, \pi) \in V_i \times U_i$ $(1 \leq i \leq l)$.
The intersection $U = \bigcap_{i = 1}^l U_i$ is an open neighborhood of $0 \in \mathcal{L}(\R^n, \R^m)$ with the desired property.  
\end{proof}
\end{document}